\documentclass{amsart}

\usepackage{xcolor}

\usepackage[lite,initials,nobysame]{amsrefs}
\usepackage{amsmath,amssymb,enumerate}
\usepackage{hyperref}
\usepackage{url}

\usepackage{float,array,diagbox}
\newcolumntype{C}[1]{>{\centering\let\newline\\\arraybackslash\hspace{0pt}}m{#1}}

\numberwithin{equation}{section}

\theoremstyle{plain} 
\newtheorem{theorem}{Theorem}[section]
\newtheorem{corollary}[theorem]{Corollary}
\newtheorem{lemma}[theorem]{Lemma}

\newtheorem{remark}[theorem]{Remark}

\begin{document}

\title{Rings as unions of proper ideals}

\author[Malcolm H. W. Chen]{Malcolm Hoong Wai Chen} \address{Department of Mathematics, University of Manchester, Oxford Road, Manchester M13 9PL, United Kingdom} \email{malcolmhoongwai.chen@postgrad.manchester.ac.uk}


 \begin{abstract}
In this note, we define and investigate \emph{ideal covering numbers} of associative rings (not assumed to be commutative or unital): three invariants defined as the minimal number of proper left, right, or two-sided ideals whose union equals the ring. For every prime $p$, we construct four infinite families of rings without identity that attain the sharp lower bound $p + 1$ for ideal covering numbers, each exhibiting distinct behavior with respect to left, right, and two-sided ideal coverings. As a consequence of a result by Lucchini and Mar\'{o}ti, we also characterize all rings with ideal covering numbers three. Finally, we make several observations and propose open questions related to these invariants and the structure of rings admitting such ideal coverings.
\end{abstract}

\maketitle

\section{Introduction} \label{sec1}

It is a standard undergraduate exercise to show that a group is never the union of two of its proper subgroups, and a classical result due to Scorza \cite{scorza} shows that a group is the union of three proper subgroups if and only if it has a quotient isomorphic to $C_2 \times C_2$, the direct product of two cyclic groups of order two. Later, Cohn \cite{groupclassic} and Tomkinson \cite{groupclassic2} began to systematically study groups which are unions of $n > 3$ proper subgroups for small values of $n$.

This motivated inquiry into a broader question: When is an algebraic structure $\mathcal{S}$ the union of $n$ of its proper substructures? We say that $\mathcal{S}$ is \emph{coverable} if such a union exists, and define the \emph{covering number} $\sigma(\mathcal{S})$ as the minimal of such $n$; we write $\sigma(\mathcal{S})=\infty$ if no such cover exists. Similar covering problems have also been studied for rings \cite{ringrecent}, vector spaces \cite{vectorspace}, modules \cite{module}, loops \cite{loop}, and semigroups \cite{semigroup}. We refer the reader to the survey article \cite{survey} for more details.

The covering numbers of groups have been studied extensively, though many questions remain open (see for example, \cite{grouprecent}). However, if one restricts to coverings by proper \emph{normal} subgroups, then there is an elegant result \cite{normal}: A group is the union of proper normal subgroups if and only if it has a quotient isomorphic to $C_p \times C_p$ for some prime $p$; the covering number is $p+1$ for smallest such $p$. 

In this note, we study the ring-theoretic analogue of these covering problems. Throughout, a \emph{ring} is a set $R$ equipped with binary operations $+$ and $\cdot$ such that $R^+:=(R,+)$ is an Abelian group, multiplication is associative, and both distributive laws hold. Rings are not assumed to be commutative or have a multiplicative identity. A \emph{subring} is an additive subgroup of $R$ closed under multiplication (but need not to contain identity). A \emph{left} (resp. \emph{right}) \emph{ideal} is an additive subgroup closed under left (resp. right) multiplication by elements of $R$; a \emph{two-sided ideal} is both a left and right ideal. In particular, every ideal is a subring. We adopt this convention to be consistent with the literature on ring covering problems.


The purpose of this paper is to initiate a study on covering rings by proper ideals, in direct analogy to covering groups by proper normal subgroups. We define the \emph{ideal covering numbers} of a ring $R$ as three invariants: the minimal number of proper left, right, or two-sided ideals whose union equals $R$. These are denoted as $\eta_\ell(R)$, $\eta_r(R)$, and $\eta(R)$, respectively. If no such cover of proper left, right, or two-sided ideal exists, we define the corresponding ideal covering number to be $\infty$. 

Note that each two-sided ideal cover is also both a left and right ideal cover, both of which is in turn also a subring cover, which by only considering the additive structure is a subgroup cover of $R^+$. Thus, (possibly allowing for infinity) we have the chain of inequalities
\begin{equation} \label{ineq}
    \sigma(R^+) \leq \sigma(R) \leq \eta_\ell(R),\eta_r(R) \leq \eta(R).
\end{equation} 

Our main result is as follows. 

\begin{theorem} \label{main}
For every prime $p$, there exist rings $R_1,R_2,R_3,R_4$ such that
\begin{itemize}
    \item $\eta_\ell(R_1)=\eta_r(R_1)=\eta(R_1)=p+1$,
    \item $\eta_\ell(R_2)=p+1$, whereas $\eta_r(R_2)=\eta(R_2)=\infty$,
    \item $\eta_r(R_3)=p+1$, whereas $\eta_\ell(R_3)=\eta(R_3)=\infty$,
    \item $\eta_\ell(R_4)=\eta_r(R_4)=p+1$, whereas $\eta(R_4)=\infty$.
\end{itemize}
\end{theorem}

Also as an immediate consequence of Lucchini and Mar\'{o}ti's classification of rings which are the union of three proper subrings \cite{ringclassic}, we obtain the following.

\begin{corollary} \label{cor}
    Let $\mathcal{M}_2(\mathbb{Z}_m)$ be the ring of $2 \times 2$ matrices over the ring of integers modulo $m$. Let $X$ be the subring $$\left\{\begin{pmatrix} 0 & 0 \\ 0 & 0 \end{pmatrix}, \begin{pmatrix} 2 & 0 \\ 0 & 0 \end{pmatrix}, \begin{pmatrix} 0 & 0 \\ 0 & 2 \end{pmatrix}, \begin{pmatrix} 2 & 0 \\ 0 & 2 \end{pmatrix}\right\}$$ of $\mathcal{M}_2(\mathbb{Z}_4)$ (note that the multiplication of any two elements is zero). Also let $Y$ and $Z$ be the subrings of $\mathcal{M}_2(\mathbb{Z}_2)$ defined by
    \begin{align*}
    Y&=\left\{\begin{pmatrix} 0 & 0 \\ 0 & 0 \end{pmatrix}, \begin{pmatrix} 1 & 0 \\ 0 & 0 \end{pmatrix}, \begin{pmatrix} 0 & 1 \\ 0 & 0 \end{pmatrix}, \begin{pmatrix} 1 & 1 \\ 0 & 0 \end{pmatrix}\right\}, \\
        Z&=\left\{\begin{pmatrix} 0 & 0 \\ 0 & 0 \end{pmatrix}, \begin{pmatrix} 0 & 1 \\ 0 & 0 \end{pmatrix}, \begin{pmatrix} 0 & 0 \\ 0 & 1 \end{pmatrix}, \begin{pmatrix} 0 & 1 \\ 0 & 1 \end{pmatrix}\right\}.
    \end{align*}
    Then for any ring $R$, we have
    \begin{itemize}
        \item $\eta(R)=3$ if and only if $R$ has a factor ring isomorphic to $X$.
        \item $\eta_\ell(R)=3$ if and only if $R$ has a factor ring isomorphic to either $X$ or $Y$.
        \item $\eta_r(R)=3$ if and only if $R$ has a factor ring isomorphic to either $X$ or $Z$.
    \end{itemize}
\end{corollary}

Note that any ring with identity cannot be covered by proper ideals, since any ideal containing the identity must be the entire ring. This makes the study of ideal covering numbers meaningful only in the context of rings without identity. However, unlike their unital counterparts, the classification and structure of such rings are relatively less understood. To the best of our knowledge, this note presents the first systematic investigation on ideal covers and our constructions provide explicit infinite families illustrating asymmetries between left, right, and two-sided ideal covers. We hope that this represents a first step towards fully understanding the structure of ideal coverable rings and their ideal covering numbers.

We also give a remark explaining how our construction in Theorem \ref{main} attains the sharp lower bound for ideal covering numbers, making it of particular interest.

\begin{remark} \label{rem}
To study ideal covering numbers, it suffices to consider rings of prime power order. This follows with minor modification from \cite[Corollary 2.4]{ringfinite}, we briefly sketch the argument here for completeness. Suppose that $R=\bigcup_{i=1}^{\eta(R)} S_i$ is a minimal two-sided ideal cover. By \cite[Lemmas 4.1 \& 4.4]{neumann}, each $S_i^+$ has finite index in $R^+$, and by \cite[Lemma 1]{lewin}, their intersection contains a two-sided ideal $I$ of finite index in $R$, so $R/I$ is finite. Since the images of $S_i$ mod $I$ form an ideal cover of $R/I$, and conversely any ideal cover of $R/I$ can be lifted to an ideal cover of $R$, so $\eta(R)=\eta(R/I)$. It is well known that any finite ring is isomorphic to a direct product of rings of prime power order, so $R/I=\prod_{i=1}^k R_i$ where $|R_i|=p_i^{n_i}$ for distinct primes $p_i$. It follows that $R/I$ is ideal coverable if and only some $R_i$ is ideal coverable, and $\eta(R/I)=\underset{1 \leq i \leq k}{\min} \eta(R_i)$.

We note that analogous results also hold true for $\eta_\ell(R)$ and $\eta_r(R)$. Let $R$ be a ring of order $p^n$ for some prime $p$ and integer $n$. Since $R^+$ is a finite Abelian group and thus a direct product of cyclic groups, it is coverable if and only if it has a quotient isomorphic to $C_p \times C_p$ (or equivalently, $n \ge 2$ and $R^+$ non-cyclic), in which case $\sigma(R^+)=p+1$ and so our examples attain the sharp bound in \eqref{ineq}.
\end{remark}

The paper is structured as follows. In Section \ref{sec2} we will provide explicit constructions for rings $R_1,R_2,R_3,R_4$, calculate the ideal covering numbers for each case and show that they satisfies the respective conditions listed in Theorem \ref{main}, and as a result we deduce Corollary \ref{cor}. In Section \ref{sec3} we will give some numerical data for the ideal covering numbers of rings of small order, make a few observations and suggest some future research directions.

\section{Proof of Main Results} \label{sec2}

Let $R$ be a finite ring and suppose that $R^+$ have generators $g_1,\dots,g_k$ of orders $m_1,\dots,m_k$, so we have relations $m_i g_i=0$ for every $1 \leq i \leq k$. Then the ring structure is completely determined by the $k^2$ products $g_i g_j = \sum_{t=1}^k c_{ij}^t g_t$ with $0 \leq c_{ij}^t \leq m_t-1$ for every $1 \leq i,j,t \leq k$. Following the work of Fine \cite{p2} these relations are said to be the \emph{presentation} of $R$, and we write
\begin{eqnarray*} \textstyle
    R = \left\langle g_1, \dots, g_k : m_i g_i = 0, \  g_ig_j = \sum_{t=1}^k c_{ij}^t g_t \ \text { for every $1 \leq i,j,t \leq k$} \right\rangle.
\end{eqnarray*}

Let $p$ be a prime. Consider the following rings with the given presentation:
\begin{align*}
    R_1 &= \langle a,b : pa=pb=0,\ a^2=b^2=ab=ba=0 \rangle, \\
    R_2 &= \langle a,b : pa=pb=0,\ a^2=a,\ b^2=b,\ ab=b,\ ba=a \rangle, \\
    R_3 &= \langle a,b : pa=pb=0,\ a^2=a,\ b^2=b,\ ab=a,\ ba=b \rangle, \\
    R_4 &= \langle a,b,c : pa=pb=pc=0,\ a^2 = a,\ ba = b,\ ac = c, \\
    &\hspace{2em} b^2 = c^2 = ab = bc = ca = cb = 0 \rangle.
\end{align*}

We note that $R_1$, $R_2$, and $R_3$ each correspond to rings $J$, $F$, and $E$, respectively, in the classification of rings of order $p^2$ by Fine \cite{p2}; while $R_4$ corresponds to ring $(23)$ in the classification by Antipkin and Elizarov \cite{p3}. 

We first give a lemma that allows us that characterizes the ideal covering numbers of rings of order $p^2$, which we can then apply to rings $R_1,R_2,R_3$. \pagebreak

\begin{lemma} \label{lem}
Let $R$ be a ring of order $p^2$ with $R^+ = \langle a,b \rangle \cong C_p \times C_p$ for some prime $p$. Then $R^+$ has exactly $p+1$ non-trivial proper subgroups, each isomorphic to $C_p$, namely $I_0=\langle a \rangle$, $I_p=\langle b \rangle$, and $I_j= \langle a+jb \rangle $ for every $1 \leq j \leq p-1$. Furthermore, $R$ is left (resp. right, two-sided) ideal coverable if and only if $I_i$ is a left (resp. right, two-sided) ideal for every $0 \leq i \leq p$. In such cases, $R=\bigcup_{i=0}^p I_i$ is a minimal left (resp. right, two-sided) ideal cover.
\end{lemma}

\begin{proof}
Every non-trivial proper subgroup of $R^+$ must be isomorphic to $C_p$. There are exactly $p^2-1$ elements of order $p$ in $R^+$, and each cyclic group of order $p$ has $p-1$ generators, so there are $\frac{p^2-1}{p-1}=p+1$ such subgroups. 

We can check that $I_0,I_1,\dots,I_p$ as defined are pairwise distinct (their pairwise intersections are trivial). Since $\sigma(R^+)=p+1$, it follows that $R^+=\bigcup_{i=0}^p I_i$ is the unique minimal subgroup cover of $R^+$. Now any left (resp. right, two-sided) ideal is also a subgroup of $R^+$, so $R$ is left (resp. right, two-sided) ideal coverable if and if and only if each $I_i$ is a left (resp. right, two-sided) ideal.
\end{proof}

We are now ready to prove our main results. 

\begin{proof}[Proof of Theorem \ref{main}]
Since the multiplication of any two elements in $R_1$ is zero, it follows that every additive subgroup of $R^+$ is a two-sided (and hence, both a left and right) ideal. Thus, $\eta_\ell(R_1)=\eta_r(R_1)=\eta(R_1)=p+1$ by Lemma \ref{lem}. It can be verified that each $I_i$ in Lemma \ref{lem} is a left (resp. right) ideal of $R_2$ (resp. $R_3$), but $\langle a \rangle$ is not a right (resp. left) ideal of $R_2$ (resp. $R_3$). Therefore, $\eta_\ell(R_2)=\eta_r(R_3)=p+1$, whereas $\eta_r(R_2)=\eta(R_2)=\eta_\ell(R_3)=\eta(R_3)=\infty$.

A direct computation shows that both $J_\ell=\langle b \rangle$ and $J_r=\langle c \rangle$ are two-sided ideals of $R_4$. We claim that $R_4/J_\ell \cong R_2$. This can be shown by identifying $a+J_\ell$ and $(a+c)+J_\ell$ as the generators of $R_4/J_\ell$ and check that it has the same multiplicative structure as $R_2$. Therefore, $\eta_\ell(R_4) \leq \eta(R_4/J_\ell)=\eta(R_2)=p+1$ and so $\eta_\ell(R_4)=p+1$. Similarly, $R_4/J_r \cong R_3$ and so $\eta_r(R_4) = p+1$. Lastly, to show that $\eta(R_4)=\infty$, it suffices to show that any proper two-sided ideal $J$ of $R_4$ does not contain $a$: If $a \in J$, then $b=ba \in J$, $c=ac \in J$, so by closure, $xa+yb+zc \in J$ for every $0 \leq x,y,z \leq p-1$. It follows that $J=R_4$, a contradiction.
\end{proof}

\begin{proof}[Proof of Corollary \ref{cor}]
Let $\eta_*(R)$ denote any one of $\eta_\ell(R)$, $\eta_r(R)$, or $\eta(R)$. Suppose that $\eta_*(R)=3$. Then from \eqref{ineq}, we have $\sigma(R) \leq \eta_*(R) = 3$, and since $\sigma(R) \ge 3$, it follows that $\sigma(R)=3$. By \cite[Theorem 1.2]{ringclassic}, the possible candidates of $R$ are those with a factor ring isomorphic to one of the following: \begin{align*}
V&=\left\{\begin{pmatrix} 0 & 0 \\ 0 & 0 \end{pmatrix}, \begin{pmatrix} 1 & 0 \\ 0 & 0 \end{pmatrix}, \begin{pmatrix} 0 & 0 \\ 0 & 1 \end{pmatrix}, \begin{pmatrix} 1 & 0 \\ 0 & 1 \end{pmatrix} \right\} \leq \mathcal{M}_2(\mathbb{Z}_2), \\ 
W&=\left\{\begin{pmatrix} a & 0 & 0 \\ b & a & 0 \\ c & 0 & a \end{pmatrix} : a,b,c \in \mathbb{Z}_2 \right\}, \\
&\text{or one of $X,Y,Z.$}
\end{align*}

Note that $V$ and $W$ are rings with identity, so if $R$ has a factor ring isomorphic to either of them, then $R$ is also a ring with identity, and thus not ideal coverable. 

Now $X$ corresponds to $R_1$ for $p=2$; while $Y$ and $Z$ correspond to $R_2$ and $R_3$ for $p=2$, respectively, by making the following identification.
\begin{itemize}
\item For $Y \cong R_2$ : $a \leftrightarrow \begin{pmatrix} 1 & 0 \\ 0 & 0 \end{pmatrix}$ \ and \ $b \leftrightarrow \begin{pmatrix} 1 & 1 \\ 0 & 0 \end{pmatrix}$.
\item For $Z \cong R_3$ : $a \leftrightarrow \begin{pmatrix} 0 & 1 \\ 0 & 1 \end{pmatrix}$ \ and \ $b \leftrightarrow \begin{pmatrix} 0 & 0 \\ 0 & 1 \end{pmatrix}$.
\end{itemize}
In particular, we have $\eta_\ell(X)=\eta_r(X)=\eta(X)=\eta_\ell(Y)=\eta_r(Z)=3$, whereas $\eta_r(Y)=\eta(Y)=\eta_\ell(Z)=\eta(Z)=\infty$. Now if a ring $R$ has a factor ring $R/I$ such that $\eta_*(R/I)=3$, then $\eta_*(R) \leq \eta_*(R/I) = 3$, and since $\eta_*(R) \ge \sigma(R) \ge 3$, it follows that $\eta_*(R)=3$, as required. This completes the proof.
\end{proof}

\section{Concluding Remarks} \label{sec3}

We wrote a simple program in GAP \cite{gap} to determine the ideal covering numbers of rings of small order using the \texttt{SmallRing} library, which includes all rings of order less than 16 (up to isomorphism). As discussed in Remark \ref{rem}, it suffices to consider rings of order $p^n$ for some prime $p$ and integer $n \ge 2$. The only integers less than 16 of this form are 4, 8, and 9. For each ring, we used the  built-in command \texttt{Subrings} to enumerate all of its subrings, and we determine whether they are left, right, or two-sided ideals by multiplying its elements on both sides by every element of the ring. We then exhaustively searched for minimal ideal covers. The covering numbers for rings of order $4 = 2^2$ and $9 = 3^2$ (prime squares) are summarized in Table \ref{tb1}, while the case of order $8 = 2^3$ (a prime cube) is shown in Table \ref{tb2}.

\begin{table}[h!]
\caption{Subring and ideal covering numbers for each coverable rings $R=\texttt{SmallRing}(n,j)$ of orders four and nine.} \label{tb1}
\centering
\begin{minipage}{0.45\textwidth}
\centering
\begin{tabular}{|c|c|c|c|c|}
\hline
$(n,j)$ & $\sigma(R)$ & $\eta_{\ell}(R)$ & $\eta_r(R)$ & $\eta(R)$ \\
\hline
$(4,4)$ & 3 & 3 & 3 & 3 \\
$(4,7)$ & 3 & 3 & $\infty$ & $\infty$ \\
$(4,8)$ & 3 & $\infty$ & 3 & $\infty$ \\
$(4,10)$ & 3 & $\infty$ & $\infty$ & $\infty$ \\
\hline
\end{tabular}
\end{minipage}
\hfill
\begin{minipage}{0.45\textwidth}
\centering
\begin{tabular}{|c|c|c|c|c|}
\hline
$(n,j)$ & $\sigma(R)$ & $\eta_{\ell}(R)$ & $\eta_r(R)$ & $\eta(R)$ \\
\hline
$(9,4)$ & 4 & 4 & 4 & 4 \\
$(9,7)$ & 4 & 4 & $\infty$ & $\infty$ \\
$(9,8)$ & 4 & $\infty$ & 4 & $\infty$ \\
\hline
\end{tabular}
\end{minipage}
\end{table}
\begin{table}[h!] 
\caption{Subring and ideal covering numbers for each coverable rings $R=\texttt{SmallRing}(8,j)$ of order eight, with all indices $j$ having the same covering numbers grouped in the same row.} \label{tb2}
\centering
\renewcommand{\arraystretch}{1.2}
\begin{tabular}{|c|c|c|c|l|}
\hline
$\sigma(R)$ & $\eta_\ell(R)$ & $\eta_r(R)$ & $\eta(R)$ & List of indices $j$ \\
\hline
3 & 3 & 3 & 3 & 5, 6, 8--12, 16, 18, 19, 25--28, 31, 32, 39 \\ \hline
3 & $\infty$ & $\infty$ & $\infty$ & 14, 40, 45, 48--50 \\ \hline
3 & 3 & $\infty$ & $\infty$ & 20, 29, 30, 37, 41 \\ \hline
3 & $\infty$ & 3 & $\infty$ & 15, 34, 43, 44, 47 \\ \hline
3 & 3 & 3 & $\infty$ & 36 \\ \hline
\end{tabular}
\end{table}
Note that for $p=2$ and $p=3$, the rings $\texttt{SmallRing}(p^2, j)$ with indices $j = 4, 7, 8$ correspond to the rings $R_1,R_2,R_3$ constructed in Theorem \ref{main}, respectively; while $\texttt{SmallRing}(8, 36)$ corresponds to $R_4$ for $p=2$. 

Nicholas Werner (private communication) suggested a ring $R$ of $3 \times 3$ matrices of a particular form (motivated by \cite[Example 6.1]{ringfinite}) over $\mathbb{F}_q$, the field of $q=p^d$ elements, for every prime $p$ and integer $d \ge 1$, for which it appears that $\eta_\ell(R)=q+1$. Using a similar argument, one can construct a closely related ring $R'$ with $\eta_r(R')=q+1$. It remains unclear whether one might have $\eta(R)=q+1$. \pagebreak

We conclude with a few natural open questions.
\begin{enumerate} 
\item Are ideal covering numbers always of the form $p^d+1$? If so, what are the possible values of $d$, especially for two-sided ideals?
\item What are the exact conditions under which equality or strict inequality is attained in each successive inequalities in the chain \eqref{ineq}?
\item Does every ring with ideal covering number $p+1$ necessarily have a factor ring isomorphic to one of the examples constructed in Theorem \ref{main}?
\end{enumerate}

\section*{Acknowledgements}

The author thanks Eric Swartz for generous email correspondence that helped initiate this work, as well as Angelina Chin, Ta Sheng Tan, and Nicholas Werner for their helpful discussions and feedback.

The author is grateful for the \emph{Ring-Theoretic Aspects of Lie Theory} workshop at the University of Edinburgh for providing a reflective environment in which some of these investigations began.

\end{document}